\documentclass[10pt, reqno]{amsart}
\usepackage{amssymb, amsmath, amsthm, chngcntr, enumerate, mathabx, mathrsfs, mathtools, tikz, url}
\usepackage[numbers]{natbib}

\usetikzlibrary{arrows,calc}

\numberwithin{equation}{section}
\newtheorem{thm}{Theorem}

\newtheorem{corollary}[thm]{Corollary}
\newtheorem{lemma}[thm]{Lemma}
\newtheorem{proposition}[thm]{Proposition}

\newtheorem{rmk}[thm]{Remark}

\newtheorem*{theorem*}{Theorem}

\newtheorem*{definition}{Definition}

\newcommand{\C}{\mathbb{C}}
\newcommand{\R}{\mathbb{R}}
\newcommand{\Z}{\mathbb{Z}}
\newcommand{\N}{\mathbb{N}}
\newcommand{\T}{\mathbb{T}}

\title[Variants of the inequalities of Paley and Zygmund]{Variants of the inequalities of Paley and Zygmund}

\author{Odysseas Bakas}
\address{Room 4606, James Clerk Maxwell Building, University of Edinburgh, Peter Guthrie Tait Road, Edinburgh, EH9 3FD.}
\email{o.bakas@sms.ed.ac.uk}

\date{}

\begin{document}

\maketitle

\begin{abstract} We examine versions of the classical inequalities of Paley and Zygmund for functions of several variables. A sharp multiplier inclusion theorem and variants on the real line are obtained.
\end{abstract}

\section{Introduction}

Let $\Lambda = (\lambda_n)_{n \in \N}$ be a lacunary sequence of positive integers, namely $\Lambda  \subset \N$ and  $\inf_{n \in\N} \lambda_{n+1}/ \lambda_n >1$. In \cite{Paley}, Paley proved that for every function $f$ in the Hardy space $H^1 (\T)$, $(\widehat{f}(n))_{n \in \Lambda} $ is square summable. Equivalently, by the closed graph theorem, there is a constant $C_{\Lambda}>0$ such that 
\begin{equation}\label{classicalPaley}
\big( \sum_{n \in \Lambda} |\widehat{f} (n) |^2 \big)^{1/2} \leq C_{\Lambda} \| f \|_{H^1 (\T)}.
\end{equation}
Zygmund proved in \cite{Zygmund_paper} (see also theorem 7.6 in Chapter XII of \cite{Zygmund_book}) that for every lacunary sequence $\Lambda = (\lambda_n)_{n \in \N}$ there are positive constants $A_{\Lambda}$ and $B_{\Lambda}$, depending only on the ratio $\rho_{\Lambda} = \inf_{n \in\N} \lambda_{n+1}/ \lambda_n $ of $\Lambda$, such that
\begin{equation}\label{classicalZygmund}
\big( \sum_{n \in \Lambda} |\widehat{f} (n) |^2 \big)^{1/2} \leq A_{\Lambda} \int_{\T} |f(\theta)|  \log^{1/2} (1+  |f (\theta) |)  d \theta + B_{\Lambda}.
\end{equation}

Several authors have studied variants of (\ref{classicalPaley}) and (\ref{classicalZygmund}). Regarding Paley's inequality, in 1937, in \cite{H-L}, Hardy and Littlewood proved that if $M = (m(n))_{n \in \N_0}$ satisfies the property
\begin{equation}\label{Paley_mult}
\sup_{N \in \N_0} \sum_{N \leq n \leq 2N} |m(n)|^2 < \infty,
\end{equation}
then $M = (m(n))_{n \in \N_0}$ is a multiplier from $H^1 (\T)$ to $L^2 (\T)$. In the opposite direction, Rudin proved in \cite{Rudin_Paley} that if $\chi_{\Lambda} $ is a multiplier from $H^1 (\T)$ to $L^2 (\T)$, then $\chi_{\Lambda}$ satisfies (\ref{Paley_mult}), or equivalently, $\Lambda$ satisfies Paley's inequality if and only if, $\Lambda$ can be written as a finite union of lacunary sequences. In \cite{DurenShields}, Duren and Shields extended Rudin's result showing that in fact every multiplier $M = ( m(n) )_{n \in \N_0}$ from $H^1 (\T)$ to $L^2 (\T)$ necessarily satisfies (\ref{Paley_mult}). In \cite{Oberlin}, D. Oberlin extended the aforementioned results to higher dimensions. For other variants of Paley's inequality see, e.g., \cite{Blasco}, \cite{Fournier}, \cite{Pisier_Paley}, \cite[Theorem 8.6]{Rudin_book} and \cite{Yudin}.

A remarkable extension of Zygmund's inequality was obtained by Rudin in his celebrated paper \cite{Rudin_gaps}. In particular, Rudin extended Zygmund's inequality from lacunary sequences to Sidon sets in $\Z$ in \cite{Rudin_gaps} and, in \cite{Rudin_book}, he extended Zygmund's inequality to Sidon sets in the dual of any compact abelian group. In other words, Rudin proved in \cite{Rudin_book} that if $\Lambda$ is a Sidon set in the dual of a compact abelian group $G$, then $\chi_{\Lambda}$ is a multiplier from $L \log^{1/2} L (G)$ to $L^2 (G)$. Moreover, Rudin conjectured that if $\chi_{\Lambda}$ is a multiplier from $L \log^{1/2} L (G)$ to $L^2 (G)$, where $\Lambda$ is an infinite set in the dual of $G$, then $\Lambda$ is a Sidon set. In \cite{Pisier_Sidon}, Pisier proved that this is indeed the case and moreover, one can actually obtain a characterisation of the multipliers from $L \log^{1/2} L (G)$ to $L^2 (G)$, see \cite {MarcusPisier}. Furthermore, higher-dimensional versions of Rudin's extension of Zygmund's inequality are well-known. Namely, if $\Lambda_j$ is a Sidon set in the dual of a compact abelian group $G_j$, then $\Lambda = \Lambda_1 \times \cdots \cdots \times \Lambda_n$ satisfies the following ``$n$-dimensional'' version of Zygmund's inequality
\begin{equation}\label{well-known}
\Big( \sum_{\gamma \in \Lambda} |\widehat{f} (\gamma)|^2 \Big)^{1/2} \leq A_{\Lambda}\int_G |f(x)| \log^{n/2} (1+|f(x)|) dx + B_{\Lambda},
\end{equation}
where $ G= G_1 \times \cdots \times G_n $, see e.g. \cite[Chapter VII]{Blei} or \cite[Remarque, p.~24]{Pisier_aleatoires}.

\subsection{Results and organisation of the paper} The present paper is organised as follows. In the next section we give some preliminaries and background. In section \ref{Zygmund_n} we introduce the notion of Sidon weights and then we extend (\ref{well-known}) from products of Sidon sets to products of Sidon weights. A question that arises naturally is whether all multipliers from $L \log^{1/2} L $ to $L^2$ are Sidon weights. In section \ref{P-Z} we give a negative answer to this question, based on a sharp multiplier inclusion theorem for functions defined over the torus that, broadly speaking, connects the classical inequalities of Paley and Zygmund. More precisely, in section \ref{P-Z} we show that the class of all multipliers from $H^1 (\T)$ to $L^2 (\T)$ is properly contained in the class of the multipliers from $L \log^{1/2} L (\T)$ to $L^2 (\T)$. Moreover, the inclusion is sharp in the sense that the exponent $r=1/2$ in $L \log^{1/2} L (\T)$ cannot be improved. As a corollary of this multiplier inclusion theorem, we give an example of a multiplier from $L \log^{1/2} L (\T)$ to $L^2 (\T)$ which is not a Sidon weight. In section \ref{real_line} we obtain an analogous inclusion theorem for functions defined on the real line by using a result of Tao and Wright on a Littlewood-Paley characterisation of functions in $L \log^{1/2} L$ with mean zero.
In the last section we study some further variants of Zygmund's inequality in higher dimensions. In particular, we obtain higher-dimensional extensions of a classical result of Bonami \cite{Bonami} and as a corollary of our results we get a special case of (\ref{well-known}) for products of lacunary sequences in $\Z$.
\subsection*{Acknowledgement} The author would like to thank and acknowledge his PhD supervisor, Professor Jim Wright, for his continuous help, support and guidance on this work and for all his useful comments and suggestions that improved the presentation of this paper.
\section{Notation and background} 
We denote by $\N$ the set of positive integers and by $\N_0$  the set of non-negative integers.

The cardinality of a set $A$ is denoted by $\#\{A\}$.

The class of all intervals of the form $\pm [2^k, 2^{k+1})$, $k \in \Z$ will be denoted by $\mathcal{I}$.

For $n \geq 2$, we use the notation $\underline{x} = (x_1, \cdots, x_n)$ for elements in $n$-dimensional euclidean space $\R^n$.

Let $G$ be a compact abelian group whose dual is denoted by $\widehat{G}$. Let $X$ be a subspace of $L^1 (G)$. We say that $m \in \ell^{\infty} (\widehat{G})$ is a multiplier from $X$ to $L^2 (G)$ if and only if, for every $f \in X$ one has
$$ \sum_{\gamma \in \widehat{G}} |m(\gamma) \widehat{f} (\gamma)|^2 < \infty.$$
The class of all multipliers from $X$ to $L^2 (G)$ will be denoted by $\mathcal{M}_{X \rightarrow L^2 (G)}$. 

The expression $X \lesssim Y$ means that there exists a positive constant $ C$ such that  $X \leq CY $. To specify the  dependence of this constant on additional parameters e.g. on $\alpha$ we  write $ X \lesssim_{\alpha} Y$. If $X \lesssim Y$ and $Y \lesssim X$, we write $X \sim Y$.

\subsection{Thin sets in Harmonic Analysis}\label{thin_sets}
Let $G$ be a compact abelian group and let $\Lambda$ be a non-empty set in its dual $\widehat{G}$. We say that a trigonometric polynomial $f$ on $G$ is a $\Lambda$-polynomial if and only if, $\widehat{f} (\gamma) = 0 $ for all $\gamma \in \widehat{G} \setminus \Lambda$.

Motivated by a classical result of Sidon \cite{Sidon}, Rudin defined in \cite{Rudin_gaps} the notion of Sidon sets (see also \cite{Rudin_book}). More specifically, an infinite set $\Lambda$ in the dual of a compact abelian group $G$ is said to be a Sidon set if and only if there is an absolute constant $S_{\Lambda} >0$ such that
\begin{equation}\label{Sidon}
 \sum_{\gamma \in \Lambda} |\widehat{f} (\gamma)| \leq S_{\Lambda} \| f\|_{L^{\infty}(G)}
\end{equation}
for every $\Lambda$-polynomial $f$. The smallest constant $S_{\Lambda}$ such that (\ref{Sidon}) holds is called the Sidon constant of $\Lambda$. Note that if $\Lambda$ is a Sidon set, then for every $\Lambda$-polynomial $f$ one automatically has
\begin{equation}\label{Rider}
  \sum_{\gamma \in \Lambda} |\widehat{f} (\gamma)| \leq R_{\Lambda} [| f |] 
\end{equation}
 where $[| f |] = \mathbb{E}\big[   \Big\| \sum_{\gamma \in \widehat{G}} r_{\gamma}  \widehat{f} (\gamma) \gamma \Big\|_{L^{\infty}(G)} \big]$ and $(r_{\gamma})_{\gamma}$ denotes the set of Rademacher functions. 
A classical result of Rider \cite{Rider} shows that the converse also holds, namely if $\Lambda$ satisfies (\ref{Rider}), then it is a Sidon set.

Let $p >2$. We say that $\Lambda \subset \widehat{G}$ is a $\Lambda(p) $ set if and only if, there exists a constant $A(\Lambda, p) >0$ such that
 $$\| f \|_{L^p (G)} \leq A(\Lambda,p) \| f \|_{L^2 (G)} $$
for every $\Lambda$-polynomial $f$. 

By Rudin's extension of Zygmund's inequality \cite{Rudin_book} and Pisier's characterisation of Sidon sets \cite{Pisier_Sidon}, a set $\Lambda$ is Sidon if and only if, for each $p>2$, $\Lambda$ is a $\Lambda(p)$ set with $A(\Lambda,p) \leq A(\Lambda) p^{1/2}$, where $A(\Lambda)>0$ is a constant that does not depend on $p$. For other proofs of Pisier's theorem, see \cite{Bourgain_Sidon} and \cite{Sidonicity}. For more details on Sidon sets, see the book \cite{Sidon_book}.

\subsection{Hardy spaces}\label{Hardy_intro}The (real) Hardy space $H^1 (\R)$ is defined to be the space of all integrable functions $f$ on $\R$ such that $H (f) \in L^1 (\T)$, where $ H (f)$ is the Hilbert transform of $f$. One defines $\| f \|_{H^1 (\R)} = \| f \|_{L^1 (\R)} + \| H (f) \|_{L^1 (\R)}$. The (real) product Hardy space on $H^1_{\mathrm{prod}} (\R^2)$ is defined as the space of all $f \in L^1 (\R^2)$ such that $H_1 (f), H_2 (f), H_1 \otimes H_2 (f) \in L^1 (\R^2)$, where $H_i$ denotes the Hilbert transform in the $i$-th variable. Similarly one defines $H^1_{\mathrm{prod}} (\R^n)$. One can define real Hardy spaces in the periodic setting in a similar way.

Let $\eta$ denote an even function supported in $\pm [1,3]$ such that $\eta|_{[3/2,2]} \equiv 1$ and $\eta$ is affine on $[1,3/2]$ and on $[3/2,2]$. It is known \cite{Stein_multiplicateurs} that $H^1 (\R)$ admits a square function characterisation, namely $\| f \|_{H^1 (\R)} \sim \| S  (f) \|_{L^1 (\R)} $, where 
$$ S (f) (x) = \big( \sum_{k \in \Z} |\Delta_k (f) (x)|^2 )^{1/2}$$ 
and $\Delta_k (f) $ is given by $ \widehat{\Delta_k} (f) (\xi) =  \eta(2^{-k} \xi) \widehat{f} (\xi) $. An analogous square function characterisation holds in the periodic setting.

\section{Higher dimensional variants of Zygmund's inequality}\label{Zygmund_n} 

In this section we examine weighted versions of (\ref{well-known}). More specifically, given compact abelian groups $G_1, \cdots, G_n$, we shall obtain a class of multipliers from $L \log^{n/2} L (G_1 \times \cdots \times G_n)$ to $L^2 (G_1 \times \cdots \times G_n)$ that properly contains multipliers of the form $\chi_{\Lambda_1 \times \cdots \times \Lambda_n}$, where $\Lambda_j $ is a Sidon set in the dual of $G_j$ ($j=1, \cdots, n$). We begin by defining the notion of Sidon weights which is a weighted analogue of the notion of Sidon sets.

\begin{definition}[Sidon weights]
Let $G$ be a compact abelian group. 

A function $m : \widehat{G} \rightarrow \C$ is said to be a Sidon weight on $G$ if and only if there is a positive constant $S_m$ such that
\begin{equation}\label{wSidon}
 \sum_{\gamma \in \mathrm{supp}(m)} | m(\gamma ) \widehat{f} (\gamma) |  \leq S_m \| f \|_{L^{\infty} (G)}
\end{equation}
for every trigonometric polynomial $f$ on $G$ whose Fourier transform is supported in $\mathrm{supp}(m)$.
\end{definition}

Note that, by (\ref{wSidon}), every Sidon weight is a bounded function on $\widehat{G}$. Moreover, if $\Lambda$ is a Sidon set in the dual of $G$, then every bounded function supported in $\Lambda$ is a Sidon weight. Therefore, the notion of Sidon weights extends that of Sidon sets.

As it is mentioned in the introduction, in \cite{MarcusPisier} it is shown that a bounded function $m$ on $\widehat{G}$ is a multiplier from $L^2 (G)$ to $\exp L^2 (G) $, or equivalently it is a multiplier from $L \log^{1/2} L (G)$ to $L^2 (G)$, if and only if,
\begin{equation}\label{Pisiermult} \sum_{\gamma \in \widehat{G}} |m(\gamma) \widehat{f} (\gamma)| \leq C _m [| f |]
\end{equation}
for all $f \in C_{\mathrm{a.s.}} (G) = \{ f \in L^2 (G) : \sum_{\gamma \in \widehat{G}} r_{\gamma} (\omega) \widehat{f} (\gamma) \gamma \ \mathrm{is}\ \omega-\mathrm{almost} \ \mathrm{surely} \ \mathrm{in} \ C(G)\},$ where 
$(r_{\gamma})_{\gamma}$ denotes the set of Rademacher functions and $[| f |]$ is as in section \ref{thin_sets}. For more details see \cite[Chapter XI]{MarcusPisier} (and in particular \cite[Corollary XI.1.5]{MarcusPisier}). It is clear that every Sidon weight on $\widehat{G}$ automatically satisfies (\ref{Pisiermult}) and hence, Sidon weights are multipliers from $L \log^{1/2} L (G)$ to $L^2 (G)$. As it is mentioned in section \ref{thin_sets}, in the unweighted setting, a classical result of Rider \cite{Rider} asserts that if for every $\Lambda$-polynomial $f$ one has
$$ \sum_{\gamma \in \Lambda} |\widehat{f} (\gamma)| \leq R_{\Lambda} [| f |], $$  
then $\Lambda$ is a Sidon set. Therefore, the following question arises. \emph{Does Rider's result hold in the weighted setting? In other words, is it true that every multiplier from $L \log^{1/2} L (G)$ to $L^2 (G)$ is a Sidon weight?} In the next section we will see that the answer to this question is \emph{no}.

In the rest of this section we focus on higher-dimensional variants of (\ref{well-known}). By adapting the arguments of Rudin \cite{Rudin_gaps} that extend Zygmund's inequality to Sidon sets, one can show that Sidon weights are multipliers from $L \log^{1/2} L (G)$ to $L^2 (G)$ without appealing to the aforementioned characterisation of the class $\mathcal{M}_{L \log^{1/2} L (G) \rightarrow L^2 (G)}$. 
Indeed, towards this aim, the first step is to obtain the following proposition, which is a weighted version of a well-known characterisation of Sidon sets \cite[Theorem 5.7.3]{Rudin_gaps}. We omit the proof as it is a straightforward adaptation of the corresponding one given by Rudin.
\begin{proposition}[Characterisation of Sidon weights]\label{weight_char}
Let $G$ be a compact abelian group and let $m: \widehat{G} \rightarrow \C$ be a function. Put $\Lambda_m = \mathrm{supp}(m)$. 

The following are equivalent: 
\begin{enumerate}
\item $m$ is a Sidon weight.
\item For every $b \in \ell^{\infty} (\Lambda_m) $ there exists a measure $\nu_b \in M(G) $ such that $\widehat{\nu_b} (\gamma) = b(\gamma) m(\gamma)$ for every $\gamma \in \Lambda_m$ and $\| \nu_b \| \leq C_m \| b \|_{\ell^{\infty} (\Lambda_m)}$, where $C_m > 0$ is a constant that depends only on $m$ and not on $b$.
\end{enumerate}
\end{proposition}
The next step is to make use of a standard adaptation of Rudin's argument to higher dimensions. For the unweighted multi-dimensional case, see e.g. \cite[Chapter VII]{Blei}. In particular, by using duality, multi-dimensional Khintchine's inequality, the above characterisation of Sidon weights and the fact that if $\nu_j \in M(G_j)$ ($j=1, \cdots, n$) then $\| \nu_1 \otimes \cdots \otimes \nu_n \| \lesssim_n \| \nu_1 \| \cdots \| \nu_n \| $, one obtains the following weighted extension of (\ref{well-known}). We omit the proof.

\begin{proposition}\label{weighted_Zyg} Let $G_1, \cdots, G_n$ be compact abelian groups. For $j =1, \cdots, n$, let $m_j : \widehat{G_j} \rightarrow \C$ be Sidon weights. 

Set $G = G_1 \times \cdots \times G_n$ and $m = m_1 \otimes \cdots \otimes m_n$. Then there are positive constants $A_{m,n}$ and $B_{m,n}$,  depending only on $m$ and $n \in \N$, such that
\begin{equation}\label{wZyg}
\Big( \sum_{\gamma \in \widehat{G} } |m (\gamma) \widehat{f} (\gamma)|^2 \Big)^{1/2} \leq A_{m,n} \int_G |f(x)| \log^{n/2 }(1+ |f(x)|) dx + B_{m,n}.
\end{equation}
In particular, $m  = m_1 \otimes \cdots \otimes m_n$ is a multiplier from $L \log^{n/2} L (G)$ to $L^2 (G)$.
\end{proposition}

\begin{rmk}\label{comment_on_Rudin}
Since Proposition \ref{weight_char} characterises Sidon weights, it doesn't seem that Rudin's argument can be adapted to the case where $m_j : \widehat{G_j} \rightarrow \C$ are just multipliers from $L \log^{1/2} L (G_j)$ to $L^2 (G_j)$, as we will see that the class of  multipliers from $L\log^{1/2} L$ to $L^2$ is strictly larger than Sidon weights. 
\end{rmk}

The converse of Proposition \ref{weighted_Zyg} is not true, even in the ``one-dimensional'' case. However, in the classical setting, namely in the unweighted case, the converse holds.

\begin{proposition}\label{product_characterisation}  Let $ n \geq 2 $ be given. Let $G_1, \cdots, G_n $ be compact abelian groups and for $j \in \{1, \cdots, n \} $, let $\Lambda_j \subset \widehat{G}_j$ be finite or countably infinite.

 Put $ G = G_1 \times \cdots \times G_n $, $ \Lambda = \Lambda_1 \times \cdots \times \Lambda_n$, and  $ S = \big\{ j \in \{ 1, \cdots, n \} : \#\{ \Lambda_j \}= \infty \big\} $. Assume that $ S \neq \emptyset$. Then, the inequality
\begin{equation}\label{ineq2}
  \Big( \sum_{\gamma  \in  \Lambda}  | \widehat{f} (\gamma) |^2  \Big)^{1/2} \leq A_{\Lambda,n} \int_G | f(x) | \log^{|S|/2} (1+ | f(x) |)  dx + B_{\Lambda,n},
\end{equation}
where $A_{\Lambda,n}$ and $B_{\Lambda,n}$ are positive constants that depend only on $\Lambda$ and on $n \in \N$, holds if and only if, $\Lambda_j$ is a Sidon set for each $j \in S$.
\end{proposition}

\begin{proof} Suppose first that $\Lambda_j$ is a Sidon set for each $j \in S$. If $S = \{ 1, \cdots, n \}$, then (\ref{ineq2}) coincides with (\ref{well-known}). So, let us assume that $ S \neq \{  1, \cdots, n \} $. Without loss of generality, we may suppose that $ S = \{ 1, \cdots, k \} $, $ k < n $. That is, $ \Lambda_1, \cdots,\Lambda_k $ are infinite countable sets, whereas the sets $\Lambda_{k+1}, \cdots, \Lambda_n $ are finite. Set $\Lambda_S = \Lambda_1 \times \cdots \times \Lambda_k $. By duality, it is enough to show that for every $\Lambda$-polynomial $f$ one has
\begin{equation}\label{LpS}
\| f \|_{L^p (G) } \lesssim p^{|S|/2} \| f \|_{L^2 (G) } 
\end{equation}
where the implied constant depends only on $\Lambda_1, \cdots, \Lambda_n$ and not on $f$. 
The main idea is to prove (\ref{LpS}) first for the special case of $\Lambda_S \times \{ \gamma'_{k+1} \} \times \cdots \times \{ \gamma'_n \}$-polynomials, where $ \{ \gamma'_j \}  $ is an arbitrary element of $\Lambda_j$, $j \in \{ k+1, \cdots, n \}$. Fix $\gamma'_j \in \Lambda_j$, $j \in \{k+1, \cdots, n \}$. If we consider a $\Lambda_S \times \{ \gamma'_{k+1} \} \times \cdots \times \{ \gamma'_n \}$-polynomial $f$, then, by using (\ref{well-known}), one can easily check that (\ref{LpS}) holds for $f$.
Observe now that every $\Lambda$-polynomial can be written as a sum of at most $(\#\{\Lambda_{k+1}\}) \cdots (\#\{\Lambda_n\})$ $\Lambda$-polynomials of the special form studied in the previous step. Therefore, by using the triangle inequality one deduces that
$$ \| f \|_{L^p ( G ) } \leq (\#\{ \Lambda_{k+1} \}) \cdots (\#\{\Lambda_n\}) A_{\Lambda_S} p^{|S|/2} \| f \|_{L^2 ( G ) },   $$
where $A_{\Lambda_S}$ is a constant that depends only on $\Lambda_S$.

To obtain the opposite direction, by Pisier's characterisation of Sidon sets, it suffices to prove that if $f$ is a $\Lambda_j$-polynomial, then
\begin{equation}\label{Pisier_est}
\| f \|_{L^p (G_j)} \lesssim p^{1/2} \| f \|_{L^2 (G_j)} 
\end{equation} 
for all $p>2$, $j \in S$. Towards this aim, take $p>2$  and let $j \in S$ be fixed. Consider an arbitrary $\Lambda_j$-polynomial $f$.  Without loss of generality, we may assume that $S \setminus \{ j\} \neq \emptyset$. Note that if $f_l$ are $\Lambda_l$-polynomials, $l \in \{ 1, \cdots, n \} \setminus \{ j \} $ then the function $F$ on $G$ given by 
$$ F (x_1, \cdots , x_n) = f (x_j) \cdot \Big( \prod_{ l \in \{ 1, \cdots, n \} \setminus \{ j \} } f_l (x_l) \Big) $$
is a $\Lambda$-polynomial.  We define the  $\Lambda_l$-polynomials $f_l$ as follows,
\begin{itemize}
\item if $l \in S \setminus \{ j\} $, then choose $f_l$ to be a $\Lambda_l$-polynomial, which satisfies 
$$ \big\| f_l \big\|_{L^p (G_l)} \gtrsim p^{1/2} \big\| f_l \big\|_{L^2 (G_l)}. $$
This is possible thanks to a construction due to Rudin \cite[Theorem 3.4]{Rudin_gaps}.
\item If $l \notin S$, then put $f_l (x_l) =  \gamma'_l (x_l) $ for some $\gamma'_l \in \Lambda_l$. In that case, for all $q>0$, $\big\| f_l \big\|_{L^q (G_l)} = 1$.
\end{itemize}
Next, note that for each $ q > 0 $ one has
\begin{equation}\label{formula}
\| F \|_{L^q (G)} = \| f \|_{L^q (G_j)} \cdot \Big( \prod_{ l \in S \setminus \{ j \} } \| f_l \|_{L^q (G_l)} \Big)
\end{equation}
Since $F$ is a $\Lambda$-polynomial, it follows by hypothesis that
$$ \| F \|_{L^p (G)} \lesssim p^{|S|/2} \| F \|_{L^2 (G)}$$
and so, by (\ref{formula}) for $q=p$ and $q=2$ one obtains
$$  \| f \|_{L^p (G_j)} \cdot \Big( \prod_{ l \in S \setminus \{ j \} } \| f_l \|_{L^p (G_l)} \Big) \lesssim p^{|S|/2} \| f \|_{L^2 (G_j)} \cdot \Big( \prod_{ l \in S \setminus \{ j \} } \| f_l  \|_{L^2 (G_l)} \Big). $$
By our construction
$$ \prod_{ l \in S \setminus \{ j \} } \| f_l \|_{L^p (G_l)} \gtrsim  p^{ (|S| - 1)/2 } \cdot \prod_{ l \in S \setminus \{ j \} } \| f_l \|_{L^2 (G_l)} $$
and so, it follows that
$$  p^{ (|S| - 1)/2 }  \| f \|_{L^p (G_j)} \lesssim p^{|S|/2} \| f \|_{L^2 (G_j)}  $$
and hence, (\ref{Pisier_est}) holds. Therefore, by Pisier's characterisation of Sidon sets, $\Lambda_j$ is a Sidon set and the proof is complete.
\end{proof}
\section{A multiplier inclusion theorem}\label{P-Z}
By Rudin's characterisation of spectral sets satisfying classical Paley's inequality \cite{Rudin_Paley} it follows that $\chi_{\Lambda} \in \mathcal{M}_{H^1 (\T) \rightarrow L^2 (\T)}$ if and only if, $\Lambda$ is a finite union of lacunary sequences. Since finite unions of lacunary sequences are Sidon sets, one deduces that $\chi_{\Lambda} \in \mathcal{M}_{H^1 (\T) \rightarrow L^2 (\T)} $ implies that $\chi_{\Lambda} \in \mathcal{M}_{L \log^{1/2} L (\T) \rightarrow L^2 (\T)}$.  Motivated by this observation, one can naturally ask whether the class  
$\mathcal{M}_{H^1 (\T) \rightarrow L^2 (\T)} $ is contained\footnote{Note that since $\| H(f) \|_{L^1(\T) } \lesssim 1 + \int_{\T} |f| \log (1+|f|)$, where $H$ is the periodic Hilbert transform, one deduces that $L \log L (\T) \subset H^1 (\T)$ and hence, one trivially has $\mathcal{M}_{H^1 (\T) \rightarrow L^2 (\T)} \subset \mathcal{M}_{L \log L (\T) \rightarrow L^2 (\T)}$.} in the class $\mathcal{M}_{L \log^{1/2} L (\T) \rightarrow L^2 (\T)}$. Our main goal in this section is to show that this is indeed the case.

\begin{thm}\label{mult_inclusion1}
The class of all multipliers from $H^1 (\T )$ to $L^2 (\T)$ is contained in the class of all multipliers from $L \log^{1/2} L (\T) $ to $L^2 (\T)$, i.e. 
\begin{equation}\label{mult_1d}
\mathcal{M}_{H^1 (\T ) \rightarrow L^2 (\T )} \subset \mathcal{M}_{L \log^{1/2} L (\T)\rightarrow L^2 (\T)} 
\end{equation}
and the inclusion is proper.
\end{thm} 

\begin{proof}
It follows by the work of Hardy and Littlewood \cite{H-L} and the work of Duren and Shields \cite{DurenShields} that $M = (m(n))_{n \in \Z}$ belongs to the class $\mathcal{M}_{H^1 (\T ) \rightarrow L^2 (\T )}$ if and only if,
\begin{equation}\label{cond_1d}
 \sup_{N \in \N} \sum_{N \leq |n| \leq 2N} |m(n)|^2 < \infty. 
\end{equation}
To prove our theorem the main idea is that one can rule out the multipliers in $\mathcal{M}_{H^1 (\T ) \rightarrow L^2 (\T )}$ with ``large'' support in the sets of  the form $I_K=  \pm [K,2K) \cap \Z$ (for example $m(n) = 1/\sqrt{|n|}$ for $n \neq 0$) and focus only on multipliers of the form $M = \chi_{\Lambda}$, where $\Lambda$ is a subset of integers satisfying $ \sup_{K \in \N} \#  \{ \pm [K,2K) \cap \Lambda \} \lesssim 1$ which can be handled by classical Zygmund's inequality.

To be more specific, let $M = (m(n))_{n \in \Z}$ be a given multiplier from $H^1 (\T)$ to $L^2 (\T)$. We may assume without loss of generality that $\mathrm{supp}(M) \subset \N_0$. We need to show that for every $f \in L \log^{1/2} L (\T)$ one has  
$$ \sum_{n \in \N_0} |m(n) \widehat{f}(n)|^2 < \infty . $$
For this, fix an arbitrary function $f \in L \log^{1/2} L (\T)$ and notice that
$$
\Big( \sum_{n \in \N_0} |m(n) \widehat{f}(n)|^2 \Big)^{1/2} = \Big(  \sum_{k \in \N_0} \sum_{2^k-1 \leq n \leq 2^{k+1} - 2  }  |m(n) \widehat{f}(n)|^2 \Big)^{1/2} $$
is majorised by
$$  \Big( \sum_{k \in \N_0} \max_{2^k -1 \leq n \leq 2^{n+1} - 2 } |\widehat{f} (n)|^2 \sum_{2^k - 1 \leq n \leq 2^{k+1} - 2  }  |m(n) |^2 \Big)^{1/2} .$$
Since $M = (m(n))_{n \in \Z}$ is a multiplier from $H^1 (\T)$ to $L^2 (\T)$ we have by (\ref{cond_1d}),
$$ A_M = \sup_{k \in \N_0} \sum_{2^k-1 \leq n \leq 2^{k+1} - 2  }  |m(n) |^2 < \infty $$
and hence,
$$ \Big( \sum_{n \in \N_0} |m(n) \widehat{f}(n)|^2 \Big)^{1/2} \leq A_M^{1/2} \Big( \sum_{k \in \N_0} \max_{2^k - 1 \leq n \leq 2^{k+1} - 2 } |\widehat{f} (n)|^2 \Big)^{1/2} . $$
Hence, it suffices to prove that
\begin{equation}\label{estimate}
 \sum_{k \in \N_0} \max_{2^k -1  \leq n \leq 2^{k+1} - 2 } |\widehat{f} (n)|^2  < \infty. 
\end{equation}
For $k \in \N_0$, denote the ``dyadic'' interval of integers $[ 2^k - 1, 2^{k+1} - 2 ] \cap \N_0 $ by $I_k$. For each $k \in \N_0$ choose a $\lambda_k \in I_k$ such that 
$$\max_{n \in I_k} |\widehat{f} (n)| = |\widehat{f} (\lambda_k)| .$$
In such a way we construct a sequence of positive integers $(\lambda_k)_{k \in \N_0}$, depending on $f$, such that
\begin{itemize}
\item $\lambda_k \in I_k$ for every $k \in \N_0$ and
\item $\sum_{k \in \N_0} \max_{n \in I_k } |\widehat{f} (n)|^2 = \sum_{k \in \N_0} |\widehat{f} (\lambda_k)|^2. $
\end{itemize}
Therefore, it is enough to show that $ \sum_{k \in \N_0} |\widehat{f} (\lambda_k)|^2 < \infty$. Notice that the first property listed above does not necessarily imply that $(\lambda_k)_{k \in \N_0}$ is lacunary and so one cannot make use of Zygmund's inequality directly. However, if we decompose $(\lambda_k)_{k \in \N_0} = \Lambda_1 \cup \Lambda_2$, where $\Lambda_1 = (\lambda_{2k})_{k \in \N_0}$ and $\Lambda_2 = (\lambda_{2k+1})_{k \in \N_0}$, then $\Lambda_1$ and $\Lambda_2$ are lacunary sequences with $2 \leq \rho_{\Lambda_i} \leq  16$, $i=1,2$. We thus deduce by (\ref{classicalZygmund}) that
$$  \sum_{k \in \N_0} |\widehat{f} (\lambda_n)|^2  = \sum_{n \in \Lambda_1} |\widehat{f} (n)|^2  + \sum_{n \in \Lambda_2} |\widehat{f}(n)|^2 < \infty $$
and this completes the proof of the inclusion (\ref{mult_1d}).

To prove that the inclusion is proper, take a Sidon set $\Lambda$ in $\Z$ which cannot be written as a finite union of lacunary sequences, see \cite[Remark 2.5(3)]{Rudin_gaps}. Then, by Rudin's characterisation of spectral sets satisfying Paley's inequality and Rudin's extension of Zygmund's inequality, $\chi_{\Lambda} \in \mathcal{M}_{L \log^{1/2} L (\T) \rightarrow L^2 (\T)} \setminus \mathcal{M}_{H^1 (\T) \rightarrow L^2 (\T)}$.
 \end{proof}

\begin{rmk} If $\Lambda$ is a lacunary sequence of positive integers with ratio $\rho_{\Lambda} \geq 2$, then the Sidon constant $S_{\Lambda}$ of $\Lambda$ is independent of $\rho_{\Lambda}$, see, e.g., \cite{Sidon_book}. Also, it can be shown that the constants $A_{\Lambda}$ and $B_{\Lambda}$ in (\ref{classicalZygmund}), actually, depend only on $S_{\Lambda}$ and hence, if $\rho_{\Lambda} \geq 2$, the constants $A_{\Lambda}$ and $B_{\Lambda}$ can be taken to be independent of $\Lambda$. 
Therefore, the argument in the proof of the theorem above implies, in fact, that if $M = (m(n))_{n \in \Z}$ is a multiplier from $H^1 (\T)$ to $L^2 (\T)$, then
$$ \Big( \sum_{n \in \Z} |m(n) \widehat{f} (n)|^2 \Big)^{1/2} \leq C_M \big[ 1 + \int_{\T} |f| \log^{1/2} (1+ |f|) \big], $$
where $C_M>0$ is a constant that depends only on $M = (m(n))_{n \in \Z}$. 
\end{rmk}

The multiplier inclusion (\ref{mult_1d}) proved above is sharp in the sense that the Orlicz space $L \log^{1/2} L (\T)$ cannot be improved. There are several ways to see this. For instance, assume that every multiplier from $H^1(\T)$ to $L^2 (\T)$ is a multiplier from $L \log^r L (\T)$ to $L^2 (\T)$ for some $r \geq  0$. We need to show that $r \geq 1/2$. For this, let $N$ be a large positive integer to be chosen later. Let $V_{2^N} = 2 K_{2^{N+1}-1} - K_{2^N -1} $ be the de la Vall\'{e}e Poussin kernel of order $2^N$, where $K_n$ denotes the Fej\'er kernel of  order $n$. Since for every $n \in \N$ one has $\| K_n \|_{L^1 (\T) } = 1$ and $\| K_n \|_{L^{\infty} (\T)} \lesssim n$, we obtain
$$ \int_{\T} |V_{2^N} (\theta)| \log^r (1+ |V_{2^N} (\theta)|) d \theta \lesssim N^r . $$
Take $M = (m(n))_{n \in \Z}$ to be $m(n)=1/\sqrt{|n|}$ for $n \neq 0$ and $m(0)=0$. Then $M $ satisfies (\ref{cond_1d}) and so, it is a multiplier from $H^1 (\T)$ to $L^2 (\T)$. Hence, we necessarily have that
$$ \Big( \sum_{n \in \Z} |m (n) \widehat{V_{2^N}}(n) |^2 \Big)^{1/2} \lesssim N^r .$$
Since $\widehat{V_{2^N}} (n)= 1$ for each $|n| \leq 2^N + 1$, we have
$$ \Big( \sum_{n \in \Z} |m (n) \widehat{V_{2^N}}(n) |^2 \Big)^{1/2} \geq \Big( \sum_{1 \leq n \leq 2^N} \frac{1}{n} \Big)^{1/2} \sim N^{1/2}. $$
Therefore, if we choose $N$ to be large enough, we deduce that we must have $r \geq 1/2$.

As a corollary of Theorem \ref{mult_inclusion1}, we obtain the following inequality.
\begin{corollary} There are absolute constants $A$ and $B$ such that
$$ \Big( \sum_{ n \neq 0 } \frac{ |\widehat{f} (n) |^2 } {|n|} \Big)^{1/2} \leq A \int_{\T} | f (\theta) | \log^{1/2}(1+ | f(\theta) | )  d \theta + B.   $$
\end{corollary}

\subsection{An example of a multiplier from $L \log^{1/2}L (\T)$ to $L^2 (\T)$ which is not a Sidon weight}\label{nonSidon}

Consider the  bounded sequence $M=(m(n))_{n \in \Z}$ given by $m(n)= 1/\sqrt{|n|} $ for $n \neq 0$ and $m(0) =0$. If we take $0< \gamma <1$ and $c \in ( (\gamma+1)/2 , 1]$, then the series
$$\sum_{n \geq 2} \frac{ e^{i 2 \pi n ( \log n)^{\gamma}} }{ n^{1/2} (\log n)^c } e^{i 2\pi n x}$$ 
converges uniformly to some $f \in C(\T)$, see \cite{Ingham}. However, since $0<c \leq 1$,
$$ \sum_{n \in \Z} |m(n) \widehat{f}(n)| =\sum_{n \geq 2} \frac{1}{ n ( \log n)^c } = \infty$$
and hence,  $M=(m(n))_{n \in \Z}$ cannot be a Sidon weight. See also \cite{Paley_counterexamples} and \cite{Paley_note}.

\section{Variants of Zygmund's inequality on the real line}\label{real_line}
Our goal in this section is to prove a real-line analogue of the multiplier inclusion theorem presented in the previous section. In order to state our main result, we need to revisit Paley's inequality for functions defined on $\R$ first. Then, by using a result of Tao and Wright on a Littlewood-Paley inequality for compactly supported functions in $L \log^{1/2} L$ with mean zero, we show that essentially all non-negative measures satisfying Paley's inequality on $\R$ also satisfy a version of Zygmund's inequality for functions supported on compact sets in the real line.
\subsection{Variants of Paley's inequality on $\R$} To formulate our main result on a real-line version of Zygmund's inequality, we first examine variants of Paley's inequality on $\R$. Characterisations of the classes of multipliers from $H^p (\R)$ to $L^q (\R)$ for $0<p\leq q<\infty $ are well-known, see \cite{McCall}. However, as we will see in the next paragraph, it is more natural to state our variant of Zygmund's inequality on $\R$ in terms of measures. Hence, in this paragraph, we also study versions of Paley's inequality with respect to non-negative measures on $\R$.

\begin{definition}[Paley measures]
A non-negative measure $\mu$ on the real line is said to be a Paley measure, if and only if, 
$$\sup_{I \in \mathcal{I}} \mu(I) < \infty . $$
\end{definition}

\begin{proposition}\label{1d_Paley}
A non-negative measure $\mu$ on $\R$ satisfies
\begin{equation}\label{wPaley} 
 \| \widehat{f} \|_{L^2 (d\mu)} \leq C_{\mu} \| f \|_{H^1 (\R )} 
\end{equation}
if and only if, $\mu$ is a Paley measure. 
\end{proposition}
\begin{proof}  Assume first that $\mu$ is a Paley measure. To prove that $\mu$ satisfies (\ref{wPaley}), consider the set $I_k = (-2^{k+1}, -2^k] \cup [2^k,2^{k+1})$ for $k\in \Z$ and write
$$ \int_{\R} |\widehat{f} (\xi)|^2 d \mu(\xi) = \sum_{k \in \Z} \int_{I_k} |\widehat{f}(\xi)|^2 d\mu(\xi). $$
Note that for every $\xi \in I_k $ one has
\begin{align*} |\widehat{f} (\xi)| \leq 2 \big[ |\eta (2^{-k+1}\xi) \widehat{f} (\xi)| + |\eta(2^{-k} \xi) \widehat{f} (\xi)| \big] &\leq 2 \big[ \| \widehat{\Delta_{k-1} (f)} \|_{L^\infty (\R)} + \| \widehat{\Delta_k (f)} \|_{L^{\infty} (\R)} \big] \\
&\leq 2 \big[ \|\Delta_{k-1}(f) \|_{L^1 (\R)} + \| \Delta_k (f) \|_{L^1 (\R)} \big],
\end{align*}
where $\eta$ and $\Delta_k$ are as in section \ref{Hardy_intro}. Therefore,
$$ \int_{\R} |\widehat{f} (\xi)|^2 d \mu(\xi) \leq 8 \sum_{k \in \Z} \mu(I_k) \big( \| \Delta_{k-1}(f) \|^2_{L^1 (\R)} + \| \Delta_k(f) \|^2_{L^1(\R)} \big) .  
$$
Hence, by using our assumption that $\mu$ is a Paley measure, it follows that
$$ \Big( \int_{\R} |\widehat{f} (\xi)|^2 d \mu(\xi) \Big)^{1/2} \lesssim \big[ \sup_{I \in  \mathcal{I}} \mu(I) \big]^{1/2} \Big(  \sum_{k \in \Z}  \| \Delta_k(f) \|^2_{L^1(\R)} \Big)^{1/2} .  
$$
By Minkowski's integral inequality and the square function characterisation of $H^1 (\R)$, we deduce that
$$ \Big( \int_{\R} |\widehat{f} (\xi)|^2 d \mu(\xi) \Big)^{1/2} \lesssim_{\mu} \| f \|_{H^1 (\R)}, $$
as desired. An analogous argument was used in the proof of \cite[Theorem 1]{Oberlin}.

For the opposite direction, we shall adapt a construction of Rudin \cite{Rudin_Paley} to the euclidean setting. More precisely, suppose that $\mu$ is a non-negative measure that is not a Paley measure, namely
$$ \sup_{k \in \Z} \mu(I_k) = \infty, $$
where $I_k$ is as above. In such a case, either there exists an increasing subsequence $(k_j )_{j \in \N}$ in $\N$ such that $\mu(I_{k_j}) \rightarrow \infty$ or there exists a decreasing subsequence $(k_j )_{j \in \N}$ of negative integers such that $\mu(I_{k_j}) \rightarrow \infty$. Without loss of generality, we may assume that we have an increasing subsequence $(k_j)_{j\in \N} $ in $\N$ with $\mu(I_{k_j}) \rightarrow \infty$, and passing to a further subsequence if necessary, we may assume that $ k_{j+1} > 5 k_j $ and $\mu(I_{k_j}) \geq j^4$. Consider the function
$$  f(x) = \sum_{ j \in \N } \frac{1} {j^2}  \widecheck{\eta_{k_j}}(x), $$
where $\eta_k (\xi) = \eta(2^{-k} \xi)$, $\eta$ being as in section \ref{Hardy_intro}. Since $\| \widecheck{\eta_k} \|_{H^1 (\R)} \lesssim 1$, where the implied constant does not depend on $k$, we see that $f \in H^1(\R)$. For every $j \in \N$, we have 
$$ \int_{I_{k_j}} |\widehat{f} (\xi) |^2 d \mu(\xi) \gtrsim  j^{-4} \mu(I_{k_j}) $$ 
and therefore $  \|  \widehat{f} \|_{L^2 (d \mu)} = \infty$, completing the proof of the proposition. 
\end{proof}
\begin{rmk}
Since every function $m \in L^{\infty} (\R)$ induces an absolutely continuous, non-negative measure $\mu$ on $\R^n$ given by $d \mu (\xi) = |m(\xi)|^2 d \xi$, one deduces that $m \in L^{\infty} (\R)$ is a multiplier from $H^1 (\R)$ to $L^2 (\R)$ if and only if, $\sup_{I \in \mathcal{I}} \int_I |m(\xi)|^2 d \xi < \infty$. We thus recover \cite[Theorem A]{McCall} for the case where $p=1$ and $q=2$. Moreover, our method is different than the one used in \cite{McCall}.
\end{rmk}
We remark that the argument presented above can be adapted to the multi-parameter case in a straightforward way. We thus obtain the euclidean analogue of \cite[Theorem 1]{Oberlin}.

\begin{proposition}\label{higher_Paley}
A non-negative measure $\mu$ on $\R^n$ satisfies
$$  \| \widehat{f} \|_{L^2 (d\mu)} \leq C_{\mu,n} \| f \|_{H^1_{\mathrm{prod}} (\R^n )}  $$
if and only if, $ \sup_{I_1, \cdots, I_n \in \mathcal{I}} \mu (I_1 \times \cdots \times I_n) < \infty$.
\end{proposition}
\subsection{A real-line version of Zygmund's inequality}In the previous paragraph we obtained a real-line version of Paley's inequality based on the square function characterisation of $H^1 (\R)$. A similar argument can be used for compactly supported functions in $L \log^{1/2 } L$ with zero mean thanks to the following deep result of Tao and Wright \cite[Proposition 4.1]{TW}.
\begin{theorem*}[Tao and Wright] 
Let $K \subset \R$ be a compact set. Let $f$ be a function in $L \log^{1/2} L  (K)$ with zero integral. 

Then for every $k \in \Z$ there exists a non-negative function $F_k$ such that
$$ |\Delta_k (f) (x)| \lesssim F_k \ast \phi_k (x) $$
for all $x \in \R$ and
$$ \big\| \big( \sum_{k \in \Z } |F_k |^2 \big)^{1/2} \big\|_{L^1 (\R)} \leq A_K \big[ \int_K |f (x)|  \log^{1/2} (1 + |f(x)|) dx + 1 \big].$$
Here $\phi_k (x) = 2^k (1 + 2^{2k} |x|^2)^{-3/4}$.
\end{theorem*}
 
We are now ready to establish a real-line analogue of Theorem \ref{mult_inclusion1}.

\begin{thm}[Weighted Zygmund's inequality on $\R$]\label{weighted_Zygmund}  Let $\mu$ be a Paley measure such that $\mu([-\delta, \delta]) = 0$ for some $\delta > 0$. 

For every compact set $K \subset \R$ there is a constant $C = C (\mu, K)>0$ such that whenever $\mathrm{supp}(f) \subset K$ one has
\begin{equation}\label{eq:wZyg}
\| \widehat{f} \|_{L^2 (d \mu)} \leq C (\mu, K) \big[ \int_K |f(x)| \log^{1/2} (1 + |f(x)|)   dx + 1 \big] .
\end{equation} 
\end{thm}

\begin{proof} Let $K \subset \R$ be a fixed compact set and let $f$ be a function supported in $K$. Assume first that $\int_K f = 0$. 

The proof of (\ref{eq:wZyg}) proceeds in the same way as the proof of Proposition \ref{wPaley}. By the aforementioned result of Tao and Wright, for each $k \in \Z$ there is a function $F_k$ such that $| \Delta_k (f )| \leq F_k \ast \phi_k $ and
$$  \big\| \big( \sum_{ k \in \Z } |F_k|^2 \big)^{1/2} \big\|_{L^1 (\R)} \leq A_K \big[ \int_K |f (x)|  \log^{1/2} (1 + |f(x)|) dx + 1 \big]  , $$
Since $\| \phi_k \|_{L^1 (\R)} \lesssim 1$, it follows that $ \| \Delta_k (f) \|_{L^1 (\R)} \lesssim \| F_k \|_{L^1 (\R)} $. Hence,
\begin{align*} \| \widehat{f} \|_{L^2 ( d \mu )} = \Big(\sum_{k \in \Z} \int_{I_k} |\widehat{f} (\xi)|^2 d \mu(\xi) \Big)^{1/2} &\lesssim \big[ \sup_{I \in \mathcal{I}} \mu(I) \big]^{1/2} \Big( \sum_{k \in \Z} \| F_k  \|^2_{L^1 (\R)} \Big)^{1/2} \\
&\leq \big[ \sup_{I \in \mathcal{I}} \mu(I) \big]^{1/2}   \Big\| \big( \sum_{ k \in \Z } |F_k|^2 \big)^{1/2} \Big\|_{L^1 (\R)} \\
&\lesssim_{\mu, K}   \int_K |f(x)|  \log^{1/2} (1 + |f(x)| )  dx +  1 .
\end{align*}

We now show why we can remove the condition that $f$ has mean zero when the measure $\mu$ vanishes on a neighbourhood of $0$. For our function $f$ supported in $K$, we may assume, without loss of generality, that
$$   \int_K |f(x)| \log^{1/2} (1+ |f(x)|)   dx \leq 1.$$
Hence, if we set $I = \int_K f$, then $|I| \lesssim 1$.
Consider 
$$g(x) = f(x) - I \psi(x) ,$$
where $\psi$ is a smooth function, supported in $K$ and such that $\int \psi = 1$. Then $g$ is supported in $K$, has mean zero and 
$$ \int_K |g(x)| \log^{1/2} (1+ |g(x)|)  dx \lesssim 1 .  $$
Hence, (\ref{eq:wZyg}) holds for $g$, as $\mu$ is a Paley measure. But
$$ \| \widehat{f} \|_{L^2 (d \mu) } \leq  \| \widehat{g} \|_{L^2 (d \mu) }  + |I| \| \widehat{\psi} \|_{L^2 (d \mu) } $$
and if $\mu$ vanishes in a neighbourhood of the origin, we have
$$ \int_{\R} |\widehat{\psi} (\xi) |^2 d \mu(\xi) = \int_{|\xi|  >\delta} | \widehat{\psi} (\xi) |^2 d\mu(\xi) \leq [\sup_{\xi \in \R} |\xi \widehat{\psi} (\xi)|^2 ] \sum_{ k \geq -M} 2^{-2k} \mu (I_k) \lesssim 1,$$
where the sets $I_k$ are as in the proof of Proposition \ref{1d_Paley}. Note that $M$ depends on $\delta > 0$ and the implicit constant in the last inequality also depends on $M$. Hence, the implicit constant depends on $\delta$, i.e. on $\mu$. Thus, (\ref{eq:wZyg}) also holds for $f$.
\end{proof}

\begin{rmk} Compared to weighted Paley's inequality on $\R$, in the previous theorem we imposed the extra hypothesis that $\mu$ vanishes on a neighbourhood of $0$. To see that this condition is necessary, consider the Paley measure $d\mu(\xi) = |\xi|^{-1} d\xi$ and take $f$ to be in the class $L \log^{1/2} L (K)$ with $\widehat{f}(0) = \int_{K} f \neq 0$, for some compact set $K \subset \R$. Since $\widehat{f}$ is continuous, 
$$ \| \widehat{f} \|^2_{L^2 (d \mu)} = \int_{\R} \frac{|\widehat{f} (\xi)|^2} {|\xi|} d\xi = \infty. $$
Note that for every $g \in H^1 (\R)$, one automatically has $\int_{\R} g =0$.
\end{rmk}
Note that if $\Lambda$ is a Sidon set in $\N$ that cannot be written as a finite union of lacunary sequences, then it follows by Rudin's extension of classical Zygmund's inequality that the discrete measure $\mu = \sum_{n \in \Lambda} \delta_n$ satisfies (\ref{eq:wZyg}), but it is not a Paley measure. Here, $\delta_n$ denotes the dirac measure supported on $\{n\}$. It is an interesting problem to characterise the class of all non-negative measures $\mu$ on $\R$ satisfying (\ref{eq:wZyg}).
\begin{rmk}By adapting the argument in the proof of Theorem \ref{weighted_Zygmund} to the periodic setting, one can give an alternative proof to Theorem \ref{mult_inclusion1}.
\end{rmk}
\section{Higher dimensional extensions of Zygmund's inequality using a theorem of Bonami}\label{Bonami_n}

In this section we obtain further extensions of Zygmund's inequality for spectral sets in $\Z^n$ by using a classical theorem of Bonami.

Let $n \geq 1$ be a fixed integer. It follows by duality that (\ref{well-known}), in the case where $ G_i = \T$ ($i=1,\cdots,n$), is equivalent to the fact that for every $\Lambda_1 \times \cdots \times \Lambda_n$-polynomial $f$ one has
\begin{equation}\label{equiv_n}
\| f \|_{L^p (\T^n)} \lesssim_{\Lambda_1, \cdots, \Lambda_n} p^{n/2} \|f \|_{L^2 (\T^n)}
\end{equation}
for all $p>2$, where the implied constant depends only on $\Lambda_1, \cdots, \Lambda_n$ and not on $f$, $p$. In particular, the classical inequality of Zygmund (\ref{classicalZygmund}) is equivalent to the fact that for all $p>2$, every $\Lambda$-polynomial $f$ satisfies (\ref{equiv_n}) for $n=1$.

In what follows we shall focus on the case where $\Lambda_i$ is a lacunary sequence in $\N$ with ratio $\rho_{\Lambda_i} \geq 2$, $i=1, \cdots, n$. Consider the two-dimensional case first. In order to prove (\ref{equiv_n}) (for $n=2$), a plausible idea is to try to iterate the one-dimensional result. To be more specific, to prove (\ref{equiv_n}) in the case of the two-torus ($n=2$), consider a $\Lambda_1 \times \Lambda_2$-polynomial $f$ and write
$$  f(\theta, \phi) = \sum_{m \in \Lambda_1} f_{\phi} (m) e^{i 2 \pi m \theta}, $$
where $f_{\phi} (m) = \sum_{ n \in \Lambda_2} \widehat{f} (m,n ) e^{i 2 \pi n \phi}$. Hence, fixing $\phi \in \T$, we may regard $f(\theta, \phi)$ as a $\Lambda_1$-polynomial. By using (\ref{equiv_n}) for $n=1$ (i.e. the classical Zygmund's inequality) one deduces that for all $p>2$
\begin{equation}\label{intermediate}
 \int_{\T} | \sum_{m \in \Lambda_1} f_{\phi} (m) e^{i 2 \pi m \theta} |^p d \theta \leq A_{\Lambda_1}^p p^{p/2} \big( \sum_{m \in \Lambda_1} |f_{\phi}(m)|^2  \big)^{p/2} 
\end{equation} 
for each fixed $\phi \in \T$. Observe now that
$$  \sum_{m \in \Lambda_1}  |f_{\phi}(m)|^2 =   \sum_{n,n' \in \Lambda_2} E_{n,n'} e^{i2\pi (n-n') \phi}, $$
where $E_{n,n'} = \sum_{m \in \Lambda_1} \widehat{f} (m,n) \overline{\widehat{f} (m,n')}$. Therefore, by integrating both sides of (\ref{intermediate}) with respect to $\phi \in \T$, one deduces
$$ \| f \|^p_{L^p (\T^2)} \leq A_{\Lambda_1}^p p^{p/2} \int_{\T} \big|  \sum_{n,n' \in \Lambda_2} E_{n,n'} e^{i2\pi (n-n') \phi} \big|^{p/2} d\phi.$$
Note that in the right-hand side of the last inequality we have a trigonometric polynomial on $\T$ frequency supported in the set $\{ n -n': n,n' \in \Lambda_2 \}$. As Zygmund's inequality handles only lacunary sequences, to obtain (\ref{equiv_n}) for $n=2$, one cannot just iterate Zygmund's inequality twice. However, one can surpass this difficulty by using the following classical result of Bonami \cite[Corollaire 4]{Bonami}.

\begin{theorem*}[Bonami]
Let $\Lambda = (\lambda_n)_{n \in \N}$ be a lacunary sequence of positive integers with ratio $\rho_{\Lambda} \geq  2$. For some $k \geq 1$, consider the sumset 
$$\Lambda^{(k)} := \{ \pm \lambda_{n_1} \pm  \cdots \pm \lambda_{n_k} : n_1 > \cdots > n_k \}. $$
Then, there exists a constant $A(\Lambda, k)$ such that for every $\Lambda^{(k)}$-polynomial $f$ one has $ \| f \|_{L^p (\T)} \le A(\Lambda, k) p^{k/2} \| f \|_{L^2 (\T)} $ for all $ p > 2$.
\end{theorem*}

To see how we can employ Bonami's result to our problem of establishing (\ref{equiv_n}) for $n=2$, write
 $$ \sum_{ n, n' \in \Lambda_1 } E_{n,n'} e^{2 \pi i (n - n') \phi} = \Big( \sum_{ n = n' } + \sum_{ n < n' } + \sum_{ n >  n' } \Big) E_{n,n'} e^{2 \pi i ( n - n') \phi} $$
and note that the diagonal term satisfies
$$ \sum_{ n = n' } E_{n,n'} e^{2 \pi i (n - n') \phi} = \sum_{ (m, n) \in \Lambda_1 \times \Lambda_2} | \widehat{f} ( m , n) |^2 = \| f \|^2_{L^2 (\T^2)} . $$
Since $p >2$, the function $x \mapsto x^{p/2}$ ($ x \geq 0$) is convex and hence
\begin{align*}
&\Big| \sum_{ n, n' \in \Lambda_1 } E_{n,n'} e^{2 \pi i (n-n') \theta} \Big|^{p/2} \leq \\
 & \frac{3^{p/2}} {3} \Big( \| f \|^p_{L^2 (\T^2)} +  \Big| \sum_{ n < n' } E_{n,n'} e^{2 \pi i (n -n') \phi} \Big|^{p/2} +  \Big| \sum_{ n > n' } E_{n,n'} e^{2 \pi i (n-n') \phi} \Big|^{p/2} \Big).
\end{align*}
Thus, 
\begin{align*}
& \| f \|^p_{L^p (\T^2)} \leq A_{\Lambda_1}^p p^{p/2} \int_{\T} \Big| \sum_{ n , n' \in \Lambda_2  } E_{n,n'} e^{2 \pi i (n-n') \phi} \Big|^{p/2} d \phi \leq \\ 
& (3A_{\Lambda_1})^p p^{p/2} \Big(  \| f \|^p_{L^2 (\T^2)} + \int_{\T} \Big| \sum_{ n < n'  } E_{n,n'} e^{2 \pi i (n-n') \phi} \Big|^{p/2} d \phi + \int_{\T} \Big| \sum_{ n > n'  } E_{n,n'} e^{2 \pi i (n-n') \phi} \Big|^{p/2} d \phi \Big)
\end{align*}
and so, by using Bonami's result to bound the off-diagonal terms, the last quantity is majorised by
$$ B^p \Big( p^{p/2} \| f \|^p_{L^2 (\T^2)} + p^p \big( \sum_{ n < n' } | E_{n,n'} |^2 \big)^{p/4} + p^p \big( \sum_{ n > n' } | E_{n,n'} |^2 \big)^{p/4}  \Big) ,$$
where $B$ depends only on $\Lambda_1, \Lambda_2$. By using the Cauchy-Schwarz inequality, one gets
$$ \sum_{ n, n' \in \Lambda_2 } | E_{n,n'} |^2 \leq \| f \|^4_{L^2 (\T^2)} $$
and hence, $ \| f \|_{L^p (\T^2)} \leq B (p^{p/2} + 2 p^p)^{1/p} \| f \|_{L^2 (\T^2)} \leq 3B p \| f \|_{L^2 (\T^2)}  $.
Therefore, the proof of (\ref{equiv_n}) for $n=2$ is complete.

As one can easily observe, in fact the above method can be used to obtain variants of Zygmund's inequality for spectral sets of the form $\Lambda_1^{(k_1)} \times \cdots \times \Lambda_n^{(k_n)} \subset \Z^n$, where $\Lambda_j^{(k_j)}= \big\{  \pm \lambda_{j,n_1} \pm \cdots \pm  \lambda_{j,n_{k_j}} : \  n_1 > \cdots > n_{k_j} \big\}$ and $\Lambda_j = (\lambda_{j,n})_{n\in \N}$ is a lacunary sequence with ratio at least $2$, for all $j=1, \cdots,n$. In other words, using the above method one obtains the following higher-dimensional extension of Bonami's result.

\begin{proposition}\label{higherBon} Let $\Lambda_j = (\lambda_{j,m})_{m\in \N}$ be lacunary sequences with $\rho_{\Lambda_j} \geq 2$ for $j=1, \cdots, n$. Let $(k_1, \cdots, k_n)$ be a  given $n$-tuple of positive integers. Then, there are positive constants $A_{\Lambda^{(k_1)}_1, \cdots, \Lambda^{(k_n)}_n}$ and $B_{\Lambda^{(k_1)}_1, \cdots, \Lambda^{(k_n)}_n}$ such that
\begin{align*}
&\Big( \sum_{ (m_1, \cdots, m_n) \in \Lambda_1^{(k_1)} \times \cdots \times \Lambda_n^{(k_n)} } |\widehat{f} (m_1, \cdots, m_n)|^2 \Big)^{1/2} \leq \\
& A_{\Lambda^{(k_1)}_1, \cdots, \Lambda^{(k_n)}_n} \int_{\T^n} |f (\underline{\theta})| \log^{K_n/2} (1 +|f (\underline{\theta})| ) d \underline{\theta} + B_{\Lambda^{(k_1)}_1, \cdots, \Lambda^{(k_n)}_n},
\end{align*}
where $K_n = k_1 + \cdots + k_n$.
\end{proposition}

To prove Proposition \ref{higherBon}, the main idea is to induct on the dimension $n \in \N$. To use induction, one needs the following lemma.

\begin{lemma}\label{Bon_Lemma} Let $n \geq 1$ be a given integer. Let $E$ be a subset of $\Z^n$, such that there are constants $C_E > 0$ and $N_E \in \N$ so that for every $E$-polynomial $g$ one has  $ \| g \|_{L^p (\T^n)} \leq C_E p^{N_E/2} \| g \|_{L^2 (\T^n)}$ for every $p > 2$.

Then, for every lacunary sequence $\Lambda = (\lambda_n)_{n\in \N}$ with ratio at least $2$ and for each $k \in \N$, there are positive constants $A_{E, \Lambda, k}$ and $B_{E, \Lambda, k }$ such that
$$ \Big( \sum_{(m,n)\in E \times \Lambda^{(k)} } |\widehat{f} (m,n)|^2 \Big)^{1/2} \leq A_{E, \Lambda, k} \int_{\T^{n+1}} |f (\underline{\theta})| \log^{(N_E + k)/2} (1+ |f (\underline{\theta})| ) d \underline{\theta} + B_{E, \Lambda, k},$$
where $\Lambda^{(k)} = \big\{  \pm \lambda_{j_1} \pm \cdots \pm  \lambda_{j_k}:  j_1 > \cdots > j_k \big\}$.
\end{lemma}
\begin{proof} The proof of the lemma is a variant of the argument given above. More precisely, take an $E \times \Lambda^{(k)}$-polynomial $g$ over $\T^{n+1}$ and for $\theta  \in \T^n$, $\phi \in \T$ write
$$ g(\theta, \phi) = \sum_{m \in E} \Big(  \sum_{n \in \Lambda^{(k)}} \widehat{g} (m,n) e^{2 \pi i n \phi} \Big) e^{2 \pi i (m \cdot \theta)} ,$$
where $m \cdot \theta $ denotes the dot product of $m$ and $\theta$, i.e. $m \cdot \theta = m_1 \theta_1 +  \cdots +  m_n \theta_n$, in the case where $n > 1$. Otherwise, $m \cdot \theta $ is just the scalar multiplication of $m \in \Z$ with $\theta \in \T$. For $p >2$, using our hypothesis, one has
$$ \| g \|^p_{L^p (\T^{n+1})} \leq C_E^p p^{p N_E /2}\int_{\T} \Big( \sum_{m \in E }\Big| \sum_{n \in \Lambda^{(k)}} \widehat{g} (m,n) e^{2 \pi i n \phi}   \Big|^2 \Big)^{p/2} d \phi.$$
We write
$$ \sum_{m \in E} \Big| \sum_{n \in \Lambda^{(k)}} \widehat{g} (m,n) e^{2 \pi i n \phi}   \Big|^2  = \sum_{n,n' \in \Lambda^{(k)}} \Big( \sum_{m \in E} \widehat{g} (m,n) \overline{\widehat{g}(m,n')} \Big) e^{2 \pi i (n-n') \phi}$$
and then split the off-diagonal part of the last sum into terms of the form 
$$  \sum_{ j_1 > \cdots > j_{2k} } \Big( \sum_{m \in E} \widehat{g} (m, \pm \lambda_{j_1} \pm \cdots \pm \lambda_{j_k}) \overline{\widehat{g}(m, \pm \lambda_{j_{k+1}} \pm \cdots \pm \lambda_{j_{2k}})} \Big) e^{2 \pi i (\pm \lambda_{j_1} \pm \cdots \pm \lambda_{j_{2k}}) \phi} $$
in order to use Bonami's theorem. The diagonal term is treated as before. The number of the above subsums depends only on $k$ and not on $g, \Lambda, p$ and so, exactly as above, one shows that
$$ \| g \|^p_{L^p (\T^{n+1})} \leq D_{E, \Lambda, k}^p p^{p(N_E+k) /2}  \Big(  \sum_{ n,n' \in \Lambda^{(k)} } |E_{n,n'}|^2  \Big)^{p/4} , $$
where $E_{n,n'} =  \sum_{m \in E} \widehat{g} (m,n) \overline{\widehat{g}(m,n')} $. By using the Cauchy-Schwarz inequality, the desired estimate follows.
\end{proof}

To prove Proposition \ref{higherBon}, assume that we are given a sequence of lacunary sequences $(\Lambda_j)_{j \in \N}$ with $ \rho_{\Lambda_j} \geq 2$ and a sequence of positive integers $(k_j)_{j \in \N}$. For each $j \in \N$ form the sumsets $\Lambda^{(k_j)}_j$. We shall use induction on the dimension $n\in \N$. Note that the one-dimensional case is Bonami's theorem. Suppose now that for $n \in \N$, Proposition \ref{higherBon} holds. To obtain the $(n+1)$-dimensional case, we set $E = \Lambda_1^{(k_1)} \times \cdots \times \Lambda_n^{(k_n)}$ and $\Lambda = \Lambda_{n+1}$. By the inductive step, it follows that $E$ satisfies the assumptions of Lemma \ref{Bon_Lemma} for  $N_E = k_1 + \cdots + k_n$. Therefore, by using that lemma, the $(n+1)$-dimensional case follows at once. Hence, the proof of Proposition \ref{higherBon} is complete.

\bibliographystyle{plainnat}
\bibliography{b2}

\end{document}